\documentclass[11pt, oneside]{amsart}
\RequirePackage{amsmath}
\usepackage{amsthm}

\usepackage[utf8]{inputenc}
\usepackage{hyperref}
\usepackage{url}

\usepackage{nomencl}
\makenomenclature

\usepackage{mathtools}
\usepackage{enumerate}
\usepackage{cite}
\usepackage[shortlabels]{enumitem}
\usepackage{tikz-cd}
\usepackage{stmaryrd}

\usepackage{latexsym}
\usepackage{amsmath,amssymb,amsthm,amscd}
\usepackage{color}
\usepackage[nolabel, final]{showlabels}
\title{On the higher analytic vectors of $\mathbf{B}_e.$}
\author{Rustam Steingart}
\address{Ruprecht-Karls-Universität Heidelberg,
		Mathematisches Institut,Im Neuenheimer Feld 205, D-69120 Heidelberg}
\email{rsteingart@mathi.uni-heidelberg.de}
\date{\today}

\theoremstyle{plain}
\newtheorem{thm}{Theorem}[section]
\newtheorem{lem}[thm]{Lemma}
\newtheorem{rem}[thm]{Remark}
\newtheorem{prop}[thm]{Proposition}
\newtheorem{cor}[thm]{Corollary}

\newtheorem{ex}[thm]{Example}
\newtheorem*{cor*}{Corollary}
\newtheorem{introtheorem}{Theorem}

\theoremstyle{definition}
\newtheorem{defn}[thm]{Definition}

\newcommand{\NN}{\mathbb{N}}

\newcommand{\Gal}{\operatorname{Gal}}
\newcommand{\Hom}{\operatorname{Hom}}

\newcommand{\ZZ}{\mathbb{Z}}
\newcommand{\Z}{\mathbb{Z}}
\newcommand{\QQ}{\mathbb{Q}}
\newcommand{\Q}{\mathbb{Q}}

\newcommand{\CC}{\mathbb{C}}
\newcommand{\C}{\mathbb{C}}

\newcommand{\fB}{\mathbf{B}}

\newcommand{\id}{\operatorname{id}}

\newcommand{\Bdr}{\fB_{\mathrm{dR}}}
\newcommand{\Bcris}{\fB_{\mathrm{cris}}}

\newcommand{\BcrisE}{\fB_{\mathrm{cris},E}}

\DeclarePairedDelimiter\abs{\lvert}{\rvert}%
\begin{document}
	\maketitle
\begin{abstract}
We prove that the first derived locally analytic vectors of the subring of Fontaine's period ring $\mathbf{B}_e$ stable under the kernel of the cyclotomic character are non-zero. Subsequently we compute their analytic cohomology. We also give a description of the cokernel of the restriction of a variant of the Bloch-Kato exponential map for $\QQ_p(n)$ to locally analytic vectors in terms of derived locally analytic vectors. We also relate pro-analytic vectors with derived locally analytic vectors in condensed mathematics for regular LF-spaces.  
\end{abstract}
	\section*{Introduction}
	Let $K/\QQ_p$ be finite.
	A cornerstone of $p$-adic Hodge theory is studying the action of the absolute Galois group $G_{K}$ on (modules over) various topological algebras such as the period rings $\Bdr,\Bcris$ of Fontaine. A special role is played by the subgroup $H_{K}:=\operatorname{ker}(\chi_{cyc}\colon G_K \to \ZZ_p^\times)$ (or more generally the kernel of a Lubin--Tate character attached to a uniformiser of a finite extension). On the one hand, the field $\CC_p^{H_{K}},$ which by Ax--Sen--Tate is the completion of $K(\zeta_{p^\infty}),$ is perfectoid, which allows us to study the action of $H_{K}$ using characteristic $p$ methods by the tilting equivalence. On the other hand, the quotient $\Gamma_{K}:=G_{K}/H_{K}$ is (via $\chi_{cyc}$) isomorphic to an open subgroup of $\ZZ_p^\times$ and in particular a $p$-adic Lie-group, which allows us to study the residual action of $\Gamma_{K}$ by analytic methods (cf. \cite{bergercolmez2016theoriedesen},\cite{FX12}, \cite{colmez2016representations}, \cite{porat2024locally}). 
	It is not always the case, that the Galois cohomology can be reconstructed from the analytic cohomology  of the locally analytic vectors (see Section \ref{sec:analytic} for precise definitions).\\
	For the purpose of this introduction we define for a Banach space $V$ with a continuous action of $\Gamma_{K}$ the subspace $V^{la}\subset V$ as the elements $v\in V$ such that the orbit map $g\mapsto gv$ is locally analytic. Equivalently, one can interpret the locally analytic vectors as the $\Gamma_K$-invariants of the space of locally analytic maps $C^{an}(\Gamma_K,V)$ for a suitable action, which allows us to define the first derived locally analytic vectors $V^{R^1la}$ as the first continuous group cohomology of the latter (see \ref{sec:analytic} for details). In the present case, since the group $\Gamma_K$ is a one-dimensional Lie group, we do not need to worry about derived locally  analytic vectors beyond degree one.  We can extend these notions to colimits of Banach spaces.  It turns out, that $V^{la}$ is naturally a module over the locally analytic distribution algebra $D^{la}(\Gamma_K,\QQ_p),$ i.e., the continuous dual  of the space of locally analytic functions $C^{an}(\Gamma_K,V).$ 
	For the purpose of this introduction we define $H^i_{an}(\Gamma_K,-):=\operatorname{Ext}_{D^{la}(\Gamma_K,\QQ_p)}(\QQ_p,-)$ for $D^{la}(\Gamma_K,\QQ_p)$-modules. 
	 It was already suggested by Berger and Colmez in \cite[Section 3.3]{bergercolmez2016theoriedesen} that $((\Bcris^+)^{\varphi=p})^{H_{K}}$  witnesses the failure of exactness of locally analytic vectors. 
	Substantial work has been done by Gal Porat and Lue Pan , who showed vanishing criteria for higher derived locally analytic vectors (cf.\cite{porat2024locally, pan2021locallyanalyticvectorscompleted}). Porat computed the locally analytic vectors of $$\fB_e^{H_{K}} = (\Bcris^{\varphi=1})^{H_{K}} = \varinjlim_k t_{\QQ_p}^{-k}((\Bcris^+)^{\varphi=p^k})^{H_{K}},$$ where $t_{\QQ_p} \in \Bdr$ denotes the usual period for the cyclotomic character. He further provides general criteria for the vanishing of the higher derived locally analytic vectors, but these do not apply to the ring $\fB_e^{H_K}.$ \\
	
	We prove the following theorem:
	\begin{introtheorem}\label{thm:intro1} (see. Theorem \ref{thm:cohomologyofR1})
	\label{thm:1}
		Let $\mathbf{B}_e = \Bcris^{\varphi=1}$ then 
		$$R:=(\mathbf{B}_e)^{H_{K},R^1la} \neq 0.$$
		Furthermore 
		$$H^i_{an}(\Gamma_K,R) \cong \begin{cases}
			\QQ_p^{[K:\QQ_p]}, &\text{ if } i=0\\
		0 &\text{ if } i\geq1.
		\end{cases}$$
	\end{introtheorem}

	Formally, we will be working with locally analytic vectors in condensed mathematics and it will follow from our results in Section \ref{sec:analytic}, that their ``underlying set'' is equal to the ad-hoc notions above.  
	The idea of the proof of Theorem \ref{thm:intro1} relies on the fundamental exact sequence 
	$$0 \to \QQ_p \to \fB_e  \to \Bdr/\Bdr^+  \to 0,$$
	which tells us that $\QQ_p$ is quasi-isomorphic to the cone 
	$$C^\bullet \colon \fB_e  \to \Bdr/\Bdr^+.$$
Using condensed mathematics, we can (roughly speaking) express continuous group cohomology as a derived functor and using the vanishing of $H^i_{cts}(G_K,\Bdr/\Bdr^+)$ we can hence compute $H^i_{cts}(G_K,\QQ_p) \cong H^i_{cts}(G_K,\fB_e).$ The left hand side is well-understood and the right hand side can be rewritten as the analytic $\Gamma_K$-cohomology of the (derived) locally  analytic vectors of the continuous $H_K$-cohomology. This allows us to relate the analytic cohomology of the derived locally analytic vectors $R:=H^0(H_K,\fB_e)^{R^1-la}$ with the Galois cohomology of $G_K$ with values in $\QQ_p$ to obtain the main result. \\
Berger and Colmez show that $(\C_p^{H_K})^{la} = K(\zeta_{p^{\infty}}) $ and hence that locally analytic vectors can serve as a functorial ``decompletion'' in certain arithmetic contexts. But contrary to the case of admissible representations, this functor can fail to be exact. 
 	Subsequently we revisit the example considered in \cite{bergercolmez2016theoriedesen}  in larger generality and prove the following theorem:
 \begin{introtheorem} \label{thm:2} (see. Theorem \ref{thm:BC16} )
 	Let $K/\QQ_p$ be finite, let $H:=G_K \cap H_{\QQ_p}$ and let $\Gamma:=G_K/H.$
 	Let $n\geq 1$ and let $U_n:=(\Bcris^+)^{\varphi=p^n}.$  Then there exists a short exact sequence 
 	$$0 \to {K(\zeta_{p^{\infty}})}\llbracket t_{\QQ_p}\rrbracket/(t_{\QQ_p})^n \xrightarrow{\delta^{la}} H^1_{cts}(H,\QQ_p(n))^{la}\to H^0(H,U_n)^{R^1la} \to 0.$$	The map $\delta^{la}$ is equal to the restriction of the boundary map $$\delta \colon H^0(H,\Bdr^+/t_{\QQ_p}^n\Bdr^+) \to H^1_{cts}(H,\QQ_p(n))$$ (obtained by applying $H$-invariants to the fundamental exact sequence) to the subspace of locally analytic vectors. 
 \end{introtheorem}
 
 While the map $\delta$ is surjective and not injective, it was already observed by Berger and Colmez (in the case $n=1$) that 
$\delta^{la}$ is not surjective and that $U_n^H$ does not contain ``enough'' analytic vectors, in the sense that the locally analytic vectors of $U_n^H$ are reduced to a one-dimensional subspace and are, in particular, not dense in $U_n^H.$ 
Using that $\fB_e = \bigcup_{n\geq0} t_{\QQ_p}^{-n}U_n$ we obtain an ``explicit'' description of $R$ in the form  $$R = \varinjlim_{n} H^0(H,U_n)^{R^1la}(-n).$$
Theoretically, one could deduce Theorem \ref{thm:1} from Theorem \ref{thm:2} by passage to the colimit with some extra work, but this would not simplify the proof of Theorem \ref{thm:1} and in this way, our proof of $R\neq 0$ does not rely on the results from \cite{bergercolmez2016theoriedesen}. 
 
The above results suggest that the behaviour of the locally analytic vectors of the $H$-invariants of a general {\it{almost $\CC_p$-representation }} (introduced in  \cite{fontaine2003presque}) such as $\CC_p(n)$ or $U_n$ can vary a lot. For example $(\CC_p^H)^{la}\subset \CC_p^H$ is dense and $\CC_p^H$ has no higher derived locally analytic vectors.

In order to liberally make use of derived functors we will work within the framework of condensed mathematics. We adapt some results in the literature to the slightly more general case of regular LF-spaces (a certain class of colimits of Fréchet spaces). These results are of independent interest, as, for example, the Robba ring is a regular LF-space which is neither an LB-space (i.e. a colimit of Banach spaces) nor strict (in the sense that it is a colimit of closed immersions).  It should be possible to circumvent condensed mathematics by stating everything purely in terms of continuous group cohomology but one would have to work out the existence of the desired spectral sequences by hand. 
 
	\section*{Acknowledgements}
	I would like to thank  Laurent Berger, Anna Blanco-Cabanillas, Hui Gao, Gautier Ponsinet, Gal Porat and Otmar Venjakob for valuable discussions and comments on earlier drafts. I thank Andrew Graham and Joaquín Rodrigues Jacinto for discussions and pointing out gaps in a previous version. Lastly I would like to thank the referee for their careful reading and helpful comments. The proof of Theorem \ref{thm:cohomologyofR1} initially relied on tedious calculations with spectral sequences and the simplified proof is based on their suggestions.\\
	 This research was supported by the Deutsche Forschungsgemeinschaft (DFG) - Project number 536703837.   
	\section{Notation and preliminaries} We fix an algebraic closure $\overline{\Q_p}$ of $\QQ_p$ and denote its completion by $\CC_p.$ Any algebraic extension $F$ of $\QQ_p$ is implicitly viewed as a subfield of $\CC_p$ and we denote by $G_F:=\operatorname{Gal}(\overline{\Q_p},F).$
Let $E/\QQ_p$ be a finite extension with ring of integers $o_E$  and let $E_\infty/E$ be the Lubin--Tate extension attached to a uniformiser $\pi_E$ of $E.$ We denote by $q=p^f$ the cardinality of $o_E/\pi_Eo_E.$
For a finite field extension $K$ of $L$ we set $H_{K,E}:= \operatorname{Gal}(\overline{\Q_p}/KE_\infty)$ and $\Gamma_{K,E}:=\Gal(KE_\infty/K).$

Let $\Bdr^{+},\Bdr,\Bcris^+,\Bcris$ be the usual period rings.  We denote by $t_{E} \in \Bdr$ the period for the Lubin--Tate character $\chi_{LT} \colon \Gamma_E \xrightarrow{\cong} o_E^\times$ (cf. \cite[Proposition 9.10]{colmez2002espaces}) and we denote by $\BcrisE:= (E \otimes_{E_0}\Bcris^+)[1/t_{E}],$ which we equip with the Frobenius $\varphi_q:= \id \otimes \varphi_p^f,$ where $\varphi_p$ denotes the Frobenius operator of $\Bcris.$
Let $\fB_e:= \Bcris^{\varphi_p=1}$ (resp. $\fB_{e,E}:= \BcrisE^{\varphi_q=1}.)$
We recall that we have a short exact sequence
\begin{equation}\label{eq:fundamentalexactsequence}
	0 \to E \to \fB_{e,E} \oplus \Bdr^+ \to \Bdr \to 0.
\end{equation}
We frequently make use of \cite{farguesfontainecourbe} as a reference. 
Fargues and Fontaine construct for a non-archimedean local field $E$ and a perfect non-trivially valued field $F$ of characteristic $p$ the curve $X_{F,E}.$ We will use their results in the case where $F$ is either $\CC_p^{\flat}$ or a finite extension of $\widehat{E}_\infty^{\flat}.$ Note that the ring $B_e:=B^+[1/t]^{\varphi=\id}$ in ({\it{loc.cit.}}) is isomorphic to our $\fB_{E,e}$ although $B$ is \textit{not} $\BcrisE$ (cf. Théorème 6.5.2 ({\it{loc.cit.}}) for $E/\QQ_p$ unramified and take  section 7.5 ({\it{loc.cit.}}) into account for $E/\QQ_p$ ramified). 
By the tilting equivalence the absolute Galois group of $\widehat{E}_\infty^{\flat}$ is isomorphic to $H_{E,E}.$ When $E=\QQ_p$ we tacitly assume $\pi_{\QQ_p}=p$ such that $(\QQ_p)_\infty = \QQ_p(\zeta_{p^\infty})$ and $\chi_{LT} = \chi_{cyc}$ is the $p$-cyclotomic character. In this case we drop $E$ from the notation and write $H_K$ for $H_{K,\QQ_p}.$ For the main results of the article  we will only need the case where $E=\QQ_p,$ but some preliminary results can be of independent interest for general $E.$
 \subsection{Computation of Galois cohomology.}
 In this section we show that passing to $H_{K,E}$-invariants poses no problem.

 \begin{lem}($H_{K,E}$-cohomology, additive case) Let $K/E$ be finite and  let $H:=H_{K,E}$ and recall some classical results on the Galois cohomology of Fontaine's period rings. 
 	\label{lem:vanishingHLadd}
 	We have 
 	\begin{enumerate}[(i)]
 		\item $H^i_{cts}(H,\fB_{E,e}) = 0$ for $i>0.$
 		\item $H^i_{cts}(H,\Bdr^+) = 0$ for $i>0.$
 		\item $H^i_{cts}(H,\Bdr) = 0$ for $i>0.$
 	\end{enumerate}
 \end{lem}
 \begin{proof} The field $F:=\CC_p^{H}$ is a finite extension of $\widehat{E_\infty}$ and hence perfectoid.
 	The first point follows from \cite[Proposition 7.1.2]{farguesfontainecourbe}, which asserts $H^1(H,\fB^{\varphi=\pi_E^d} \otimes_E V)=0$ for every $d>0$ and every finite dimensional $L$-representation $V.$ Writing $\fB_{E,e} = \varinjlim t_{E}^{-d}\fB^{\varphi=\pi_E^d}$ and using $t_{E}^{-d}\fB^{\varphi=\pi_E^d} \cong \fB^{\varphi=\pi_E^d} \otimes E(\chi_{LT}^{-d}) $ yields the claim. The continuous cohomology commutes with colimits in this case because $H^n$ is compact for every $n$ and any compact subset is contained in some level of the colimit (cf. Lemma \ref{lem:regularconditions} for details on this).\\
 	The second point is  \cite[Theorem 3.1]{iovita1999galois}. The third point follows by passing to the inductive limit $\Bdr = \varinjlim_n t_E^{-n}\Bdr^+$ using that the action of $H$ on $t_E$ is trivial.
 \end{proof}
 \begin{lem}($H_{K,E}$-cohomology multiplicative case)
 	\label{lem:vanishingHLmult}
 	Let $d \in \NN.$ Let $K/E$ be finite and let $H:=H_{K,E}$. Let $W$ be a finitely generated free $\fB_{E,e}$-module with a continuous semi-linear action of $G_K.$ Then:
 	\begin{enumerate}[(i)]
 		\item  $W^{H}$ is a finite projective  $\fB_{E,e}^{H}$-module and the natural map
 		$$\fB_{E,e} \otimes_{\fB_{E,e}^{H}} W^{H} \to W$$ is an isomorphism.  
 		\item $H^1_{cts}(H,GL_d(\Bdr^+/\operatorname{Fil}^i\Bdr^+)) = 1$ for every $i>0.$
 		\item  $H^1_{cts}(H,GL_d(\Bdr^+)) = 1.$ 
 	\end{enumerate}
 \end{lem}
 \begin{proof}
 	For {\it{(i)}} see \cite[Théoreme 9.3.4]{farguesfontainecourbe}. 
 	For {\it{(ii)}} and {\it{(iii)} }	see \cite[Proposition IV 2.2]{colmez1998theorie} for a proof in the cyclotomic case using Sen's method. Since it does not make use of the continuous normalised trace map, this proof can be adapted to the Lubin-Tate situation, see also \cite[Proof of Lemma 3.1]{porat2022lubin}. 
 \end{proof}
Contrary to $\fB_{E,e},$ which is a PID, the ring $\fB_{E,e}^{H_{E,E}}$ is merely a Dedekind domain whose class group is isomorphic to $\operatorname{Hom}(H_{E,E},E^\times)$ by \cite[Proposition 7.2.4]{farguesfontainecourbe}. For this reason the formulation in {\it{(i)}} above is more cumbersome than the formulation for the PIDs $\Bdr^+$ and $\Bdr$. 
\begin{lem} \label{lem:classicalvanishing}
	Let $K/E$ be finite, let $n \in \NN_0$ and let $i>0,$ then
	\begin{enumerate}[(i)]
		\item $H^n_{cts}(G_K,\Bdr^+/\operatorname{Fil^i}\Bdr^+(-i))=0.$ 
		\item $H^n_{cts}(G_K,\Bdr/\Bdr^+)=0.$ 
	\end{enumerate} 
\end{lem}
\begin{proof}
	The second point follows from the first by passing to colimits. For the first point consider for $d>0$ the short exact sequence 
	$$0 \to \Bdr^+/\operatorname{Fil}^{{i-1}} \Bdr^+(-(i-1)) \to \Bdr^+/\operatorname{Fil}^i \Bdr^+(-i) \to \CC_p(-i) \to 0.$$
	The statement follows by induction on $i$ from the vanishing of $H^n_{\text{cts}}(G_K,\CC_p(-d))$ for every $d>0$ which holds by \cite[Proposition 3.2]{iovita1999galois}.
\end{proof}
	\section{Analytic vectors and condensed mathematics}
	In this section we recall some notions from condensed mathematics and extend them for our purposes. To keep the exposition clear we will first focus on the condensed side of things. In Section \ref{sec:clasvscon} we will explain in detail how to go back and forth between classical and condensed objects. 
	\subsection{General condensed mathematics}
	We fix some uncountable strong limit cardinal $\kappa$ and denote by $\operatorname{Cond}$ the category of $\kappa$-condensed sets in the sense of \cite[Remark 1.3]{scholze2019condensed}, i.e., the category of sheaves on the $\kappa$-small pro-étale site of a point $*.$ For a topological space $X$ we denote by $\underline{X}$ the condensed set $T \mapsto C(T,X).$ For a finite extension $K/\QQ_p$ we denote by $\operatorname{Mod}_{K}$ the category of $\underline{K}$-modules in $\operatorname{Cond}.$ We recall that to the Huber pair $(K,o_K)$ we can attach an analytic ring $K_\blacksquare$ and we denote by $\operatorname{Mod}_{K_\blacksquare} \subset \operatorname{Mod}_{K}$ the corresponding subcategory of solid objects (cf. \cite[Lecture 7]{scholze2019condensed}). Recall that a complete locally convex space $V$ over $K$ can be written as a limit of Banach spaces (namely the completions of $V$ with respect to the semi-norms defining its topology). Because the category of solid modules is stable under limits and Banach spaces are solid, we conclude that $\underline{V}$ is solid. 
	We recall the following exactness property of the evaluation at $*,$ which will be used tacitly throughout the article. 
	\begin{lem}\label{lem:exactness}
		The functor $T \mapsto T(*)$ mapping a condensed abelian group $T$ to its underlying group is exact. 
	\end{lem}
	\begin{proof}This follows because any (hyper) cover $S^\bullet \to *$ is split. 
	\end{proof}

	\subsection{Condensed functional analysis} 
	We give an overview of some results of condensed functional analysis spread in the literature. We expand on some specific colimits of Fréchet spaces of interest to us. 
	\begin{defn}
		We say $W \in \operatorname{Mod}_{K}$ is \textbf{classical} if $W$ is of the form $\underline{V}$ for a $(T1)$ topological vector space $V$ over $K.$ We say $W$ is a \textbf{solid Banach space} (resp. \textbf{solid Fréchet space}) if $W$  belongs to the essential image of $\underline{(-)}$ restricted to the category of Banach spaces (resp. Fréchet spaces) in the usual\footnote{A Banach (resp. Fréchet) space is a complete Hausdorff TVS whose topology is defined by a single (resp. a countable family of) seminorm(s).} sense. 
		A (not necessarily) classical $W \in \operatorname{Mod}_{K}$ is called \textbf{solid LF-space} if it can be written as a countable colimit of solid Fréchet spaces. 
	\end{defn}
	
	\begin{rem}
		The functors $V \mapsto \underline{V}$ and $W \mapsto W(*)$ are exact mutually inverse equivalences of categories between classical and solid Banach (resp. Fréchet) spaces. 
	\end{rem}
	\begin{proof}
		See \cite[Section 2.2.4]{colmez2023arithmeticdualitypadicproetale}.
	\end{proof}
	
	For later use let us record some properties of solid tensor products. 		We denote by $-\otimes_{K_\blacksquare}-$ (resp. $-\otimes_{K_\blacksquare}^{\mathbb{L}}-$) the (derived) solid tensor product.  
We recall that a solid vector space $V$ is said to be \textbf{solid nuclear} 
if 
the natural map
$$\underline{\Hom}_{K}(K_\blacksquare[S],K) \otimes_{K_\blacksquare}V \to \underline{\Hom}_{K}(K_\blacksquare[S],V)$$ is an isomorphism for every extremally disconnected set $S.$

 For our purposes it is enough to know that Fréchet spaces are \textbf{solid nuclear}. 
	\begin{lem}\label{lem:tensorproduct}We have 
		\begin{enumerate}[(i)]
			\item $-\otimes_{K_\blacksquare}-$ commutes with colimits in each variable. 
			\item If $V_n$ is a countable pro-system of solid nuclear spaces and $W$ is a Fréchet space, then 
			$$\varprojlim_n (V_n\otimes_{K_\blacksquare} W )= V \otimes_{K_\blacksquare}W.$$
			\item If $V_n$ is a countable pro-system of complexes of solid nuclear spaces and $W$ is a bounded above complex of Fréchet spaces, then 
			$$	\mathbf{R}\varprojlim_n (V_n\otimes^{\mathbb{L}}_{K_\blacksquare} W )= (V\otimes^{\mathbb{L}}_{K_\blacksquare} W ).$$
			\item 	If $\underline{V}$ and $\underline{W}$ are Fréchet spaces, then $\underline{V} \otimes_{K_\blacksquare} \underline{W} = \underline{V \widehat{\otimes}_KW},$ where the right hand side denotes the completed (projective \footnote{The projective tensor product topology is the strongest locally convex topology making the canonical map $V \times W \to V \otimes_KW$ jointly continuous (cf. \cite[Definition 10.3.2]{PGS}).}) tensor product.
			\item Fréchet spaces are flat on solid K-vector spaces in the sense that $\underline{V} \otimes_{K_\blacksquare} - = \underline{V} \otimes^{\mathbb{L}}_{K_\blacksquare} - .$

		\end{enumerate}
		
	\end{lem}
	\begin{proof}
		The first point follows from the fact that tensor products and solidification commute with colimits. For the rest see \cite[A.65,A.67,A.68]{bosco2023padicproetalecohomologydrinfeld}.
		
	\end{proof}
	\subsection{Condensed continuous group cohomology}
	Let $G$ be a condensed group, i.e., a group object in the category of condensed sets, and denote by $\mathbb{Z}[G]$ its condensed group ring. For a $\mathbb{Z}[G]$-module $M$ the group cohomology of $G$ with values in $M$ is defined as the derived internal Hom
	$$\mathbf{R}\Gamma_{\operatorname{cond}}(G,M):= \mathbf{R}\underline{\Hom}_{\mathbb{Z}[G]}(\mathbb{Z},M) \in D(\operatorname{Cond}(Ab)).$$
	When $M$ is an object of $\operatorname{Mod}_{K_\blacksquare}$ and $G$ is a profinite group then 
	$$\mathbf{R}\Gamma_{\operatorname{cond}}(\underline{G},M) \simeq \mathbf{R}\underline{\Hom}_{K_\blacksquare[\underline{G}]}(K,M),$$ where $K_\blacksquare[G] = (o_K\otimes_{\ZZ}\ZZ[\underline{G}])^{\blacksquare}[1/p]$ is the solidification of the group ring. (cf. \cite[Lemma B.4]{bosco2023padicproetalecohomologydrinfeld}).
One has the following comparison theorem: 
\begin{thm}
	\label{thm:comparison}. 
	Let $G$ be a profinite group and $M$ a $T1$ topological $G$-module such that $\underline{M}$ is solid. Then 
	$$H^i(\mathbf{R}\Gamma_{cond}(\underline{G},\underline{M})(*)) = H^i_{\text{cts}}(G,M).$$
\end{thm}
\begin{proof}
	See \cite[Appendix B]{bosco2023padicproetalecohomologydrinfeld}.
\end{proof}
	\subsection{Analytic vectors}
	\label{sec:analytic}
	Let $G$ be a compact $p$-adic Lie group.
	Recall that the space $C^{an}(G,K)$ of locally analytic functions $f\colon G \to K$  (cf. \cite[Definition 2.1.25]{emerton2017locally}) is an LB-space of compact type (which in particular is regular by Example \ref{ex:regular examples} below). By Lemma \ref{lem:regularLF} thus we know that $\underline{C^{an}(G,K)}$ can be written as a countable colimit of Banach spaces as well. When $V$ is a $K$-Banach space we denote by $C^{an}(G,V)$ the space of locally analytic functions $f \colon G\to V,$ consisting of functions $f \colon G \to V$ which locally around each point admit a expansion as a convergent $V$-valued power series in local coordinates of $G.$ 
The space $C^{an}(G,V)$ carries an action of $G^3$ given by $(g_1,g_2,g_3) (f(x)) = g_1f(g_3^{-1}xg_2).$
Classically one defines the analytic vectors of a (amenable) topological vector space $V$ with a $G$ action as the set of $v \in V$ such that the orbit map $G \to V$ mapping $g$ to $gv$ is locally analytic. A useful observation (cf. \cite[Defintion 3.2.8]{emerton2017locally}) is that a map $G \to V$ is of the form $g\mapsto gv$ (for some fixed $v \in V$) if and only if it is stable under the $(1,3)$-action of $G$ on $\operatorname{Map}(G,V)$ given by $(g*_{1,3}f)(x)=gf(g^{-1}x).$ Indeed, if $f$ is the orbit map of $v,$ then $f(g) = gf(1)$ and hence $gf(g^{-1}x)=gg^{-1}xf(1) =xf(1)=f(x).$ Conversely if $h\colon G \to V$ is a map stable under the $(1,3)$-action then $h(x)=xh(x^{-1}x) = xh(1)$ for every $x \in G.$  
	This motivates the algebraic definition of analytic vectors as follows:
\begin{defn}
	We denote by $(-)^{Rla}$ the functor 
	\begin{align*} D(\operatorname{Mod}_{K_\blacksquare[\underline{G}]}) &\to D(\operatorname{Cond}(Ab))\\
		V	&\mapsto \mathbf{R}\Gamma_{\operatorname{cond}}(\Gamma,C^{la}(G,V)_{*_{1,3}}),
	\end{align*}
	where $G$ acts on the right hand side via the $(1,3)$-action and $C^{la}(G,V)$ is defined as $$V \otimes_{K_\blacksquare}^{\mathbb{L}} \underline{C^{an}(G,K)}.$$ By the $(1,3)$-action we mean the action induced by the diagonal action of the action on $V$ and the $(1,3)$-action on $C^{an}(G,K)$ (with trivial action on $K$). We denote by $(-)^{R^ila}$ the $i$-th cohomology of $(-)^{Rla}.$
\end{defn}
The resulting object still carries a $G$ action induced by the previously described action of $G^3$ via the embedding $1 \times G\times 1 \to G^3.$ 
The functor $(-)^{Rla}$ commutes with small colimits (this is not entirely formal see \cite[Proposition 3.1.7. ]{jacinto2023solidlocallyanalyticrepresentations}).
 For projective limits the situation is a bit more subtle. We will revisit this issue in subsection \ref{sec:proanalytic}.

We will deviate from the notation of \cite{jacinto2023solidlocallyanalyticrepresentations} in order to emphasize which objects are considered as derived. 
\begin{defn} We denote  $\operatorname{DRep}_{K_\blacksquare}(G)$ the derived category of the category of $K_\blacksquare[G]$-modules on $K_\blacksquare$-vector spaces. We say $V \in \operatorname{DRep}_{K_\blacksquare}(G)$ is analytic if $V \cong V^{Rla}$ via the natural map. We denote by $\operatorname{DRep}^{an}_{K_\blacksquare}(G)$ the category of analytic objects. Let $D^{la}_{cond}(G,K):= \underline{\Hom}_{K}(\underline{C^{an}(G,K)},\underline{K})$ be the condensed distribution algebra. For $W\in \operatorname{DRep}^{an}_{K_\blacksquare}(G),$
following \cite{jacinto2023solidlocallyanalyticrepresentations}, we define
	 analytic cohomology as $$\mathbf{R}\Gamma^{la}(G,W):= \mathbf{R}\underline{\Hom}_{\underline{D^{la}(G,K)}}(\underline{K},W).$$
Note that the derived analytic vectors are indeed module objects over $D^{la}(G,K)$ by \cite[Theorem 4.36.]{jacinto2022solidlocallyanalyticrepresentations}.
\end{defn}

\begin{thm}\label{thm:analyticcondcomp} Let $V$ be in $\operatorname{DRep}_{K_\blacksquare}(G)$ we have 
	$$\mathbf{R}\Gamma_{\text{cond}}(G,V) = \mathbf{R}\Gamma_{\text{cond}}(G,V^{Rla})= \mathbf{R}\Gamma^{la}(G,V^{Rla}).$$
\end{thm}

\begin{proof}
	See \cite[Theorem 6.3.4]{jacinto2023solidlocallyanalyticrepresentations}
\end{proof}

\section{Classical vs. Condensed} \label{sec:clasvscon}
In this section we summarize results that allow us to pass between classical and condensed objects. 
\subsection{LF-spaces}
In general $V\mapsto \underline{V}$ does not commute with colimits so we run into problems when studying LF-spaces by which we mean countable locally convex inductive limits of Fréchet spaces with injective transition maps. 

It is customary to equip the underlying set $V(*)$ of a condensed $K$-vector space with the compactly generated topology, i.e. the finest topology making the maps $[S \to V(*)] \in V(S)$ continuous for every compact Hausdorff space $S,$ which we denote $V(*)^{\text{cg}}.$ In the context of locally convex topological vector spaces it makes more sense equip $V(*)$ with the finest locally convex topology with this property, which we denote by $V(*)^{\text{lccg}}.$ We call this the locally convex compactly generated topology. 
There is a continuous bijection $V(*)^{{\text{cg}}} \to V(*)^{\text{lccg}}.$
We say a locally convex topological vector space $V$ is \textbf{lccg} (locally convex compactly generated) if the topology on $V$ is the finest locally convex topology making the natural map $V^{cg} \to V$ continuous, where $V^{cg}$ is $V$ equipped with the compactly generated topology as a topological space. 
\begin{prop}\label{prop:lccg}
	The following hold:
	\begin{enumerate}[(i)]
		\item Metrizable locally convex spaces are lccg.
		\item If $V$ is equipped with a final locally convex topology  with respect to a family of maps $V_i \to V$ and the $V_i$ are lccg then so is $V.$ 
	\end{enumerate}
\end{prop}
\begin{proof}
	If $V$ is locally convex and metrizable, then the topology of $V$ is already compactly generated. As a consequence $V=V^{{\text{cg}}} = V^{\text{lccg}}$ when $V$ is metrizable. For the second point $V$ is equipped with the finest locally convex topology making the maps $V_i \to V$ continuous. 
	By definition we have a continuous bijection $V^{lccg} \to V.$ To check that it is bijective it suffices to see that the identity $V \to V^{lccg}$ is continuous, but to this end it suffices to check that $V_i \to V^{lccg}$ is continuous for every $i,$ for this it suffices to check that $S \to V \to V^{lccg}$ is continuous for every compact Hausdorff space but this holds by definition of the lccg topology. 
\end{proof}
Recall that an LF-space  $V= \varinjlim_n V_n$ is called \textbf{regular}, if every bounded subset of $V$ is contained in some $V_n$ and bounded in $V_n.$
\begin{lem} \label{lem:regularLF}
	Let $V  = \varinjlim_n V_n$ be a regular and complete LF-space. Then 
	$$\varinjlim{\underline{V}_n} = \underline{V}$$
	and $\underline{V}(*)^{\text{lccg}}=V$ as topological spaces. In particular, $\varinjlim_n \underline{V_n}$ is classical. 
\end{lem}
\begin{proof} First of all $\underline{V}$ is well-defined as regular inductive limits of Hausdorff spaces are Hausdorff (cf. \cite[Theorem 11.2.4]{PGS}). Let $T$ be a compact Hausdorff space and $f\colon T\to V$ a continuous map. Then $f(T)$ is bounded and hence contained in some $V_n.$ This shows that the natural map $\varinjlim_nC(T,V_n)\to C(T,V)$ is bijective, provided we can show that $T\to V_n$ is continuous. 
A priori, the subspace topology $V_n^{LF}$ induced by $V$ on $V_n$ is weaker. By increasing $n$ we can ensure that $V_n$ contains the completed convex hull $C$ of the image of $T,$ which is a complete convex compactoid. This allows us to apply \cite[Theorem 6.2.1]{PGS} to conclude that the topology $V_n^{LF}$ agrees on $C$ with the topology of $V_n,$ using that $K$ is locally compact and hence spherically complete.
	 Hence $\varinjlim \underline{V_n} = \underline{V}.$ By definition, we have $\underline{V}(*)^{lccg} \cong V$ endowed with the locally convex compactly generated topology. Hence in order to show $\underline{V}(*) \cong V$ as a topological space, we need to see that $V$ is lccg. This follows from Proposition \ref{prop:lccg}.
\end{proof}

\begin{lem}\label{lem:regularconditions}
	Let $V = \varinjlim_{n \in \NN} V_n$ be a classical LF-space, written as a colimit of Fréchet spaces with injective transition maps. Then $V$ is regular, Hausdorff and complete if one of the following holds: 
	\begin{enumerate}[(i)]
		\item The transition maps are closed embeddings.
		\item The $V_n$ are Banach spaces and the transition maps $V_n \to V_{n+1}$ are compact operators for every $n \in \NN.$
	\end{enumerate}
\end{lem}
\begin{proof}
	See \cite[Theorem 11.1.5]{PGS} and \cite[Theorem 11.3.5]{PGS}.
\end{proof}
Before providing examples of regular LF-spaces let us first highlight the special role compact LB-spaces over locally compact fields. 
\begin{lem}
	Let $V = \varinjlim V_n$ be a compact LB-space over a locally compact field. Then the locally convex inductive limit $V$ is compactly generated as a topological space and coincides with the topological inductive limit of the $V_n.$
\end{lem}
\begin{proof}
	The topological inductive limit is compactly generated, hence if we can show that the locally convex inductive limit is the topological inductive limit, we are done. This is a consequence of the analogue of \cite[TVS III.6 §1 Lemma 1.]{Bourbaki2002-qn} for locally compact non archimedean fields.
\end{proof}
We end with some examples of regular LF-spaces in $p$-adic Hodge theory.
\begin{ex}\label{ex:regular examples}
	The following LF-spaces are regular: 
	\begin{enumerate}[(i)]
		\item $\Bdr= \varinjlim_k t_{E}^{-k}\Bdr^+$ and  $H^0(H_{E,E},\Bdr)= \varinjlim_k H^0(H_{E,E},t_{E}^{-k}\Bdr^+).$
		\item  $\fB_{e,E}$ and $H^0(H_{E,E},\fB_{e,E}).$
		\item $C^{an}(G,K)$ for a compact $p$-adic Lie group $G$ and a complete field extension $K$ of $\QQ_p.$
		\item The Robba ring $\mathcal{R}_K$ of Laurent series with coefficients in a complete subfield $K$ of $\CC_p,$ which converge on a half-open interval $[r,1)$ for some $r \in (0,1).$ 
	\end{enumerate}
\end{ex}
\begin{proof}
	For $\Bdr,$  since the maps $t^{-k}_{E}\Bdr^+ \to t^{-k-1}_{E}\Bdr^+$ are closed embeddings one gets regularity of the limit by Lemma \ref{lem:regularconditions}. The same holds after taking $H_{E,E}$-invariants. 
	
	For $\fB_{e,E}$ it is convenient to represent it as 
	$\varinjlim_n t_{E}^{-n}B^{\varphi_q=\pi_L^n},$ where each term has a natural Banach space structure (cf. \cite[Section 4.1.2]{farguesfontainecourbe}).
	We claim that $t_{E}B^{\varphi_q=\pi_E^n} \to B^{\varphi_q=\pi_E^{n+1}}$ is a closed embedding. 
	This can be deduced from  the Snake Lemma applied to the diagram (obtained from variants of the fundamental exact sequence (see \cite[Exemple 6.4.2]{farguesfontainecourbe}))
	$$
	\begin{tikzcd}
		0 \arrow[r] & t_EEt_E^n \arrow[d, equal] \arrow[r] & t_EB^{\varphi_q=\pi_E^{n}} \arrow[d, hook] \arrow[r] & t_E\Bdr^+/t_E^{n+1}\Bdr^+ \arrow[d, hook] \arrow[r] & 0 \\
		0 \arrow[r] & Et_E^{n+1} \arrow[r]                             & B^{\varphi_q=\pi_E^{n+1}} \arrow[r]                & \Bdr^+/t_E^{n+1}\Bdr^+ \arrow[r]                  & 0.
	\end{tikzcd}
	$$
Which produces a short exact sequence of Banach spaces 
$$0 \to t_{E}B^{\varphi_q=\pi_E^n} \to B^{\varphi_q=\pi_E^{n+1}} \to \CC_p \to 0.$$ By a standard argument using the Open Mapping Theorem this short exact sequence is strictly exact. Regularity follows again from Lemma \ref{lem:regularconditions}.
	Since $t_{E}$ is $H_{E,E}$-invariant we get a similar decomposition for the invariants. \\
	For $C^{an}(G,K)$ see the discussion after \cite[Definition 2.1.25]{emerton2017locally}, which asserts that $C^{an}(G,K)$ is an LB-space with compact transition maps. The regularity follows by Lemma \ref{lem:regularconditions}.\\
	For the Robba ring see \cite[Proposition 2.6]{Berger}. The statement even holds for the generalised version over the character variety studied in {\it{(loc.cit)}}.
\end{proof}
\subsection{Classical and condensed analytic cohomology}
Classically, the distribution algebra $D^{la}(G,K)$ is defined as the strong dual of $C^{an}(G,K),$ i.e., the continuous dual equipped with the strong topology. A sub-basis of the strong topology on $\Hom_{K,cts}(V,K)$ is given by the sets $\mathcal{U}(B,U),$ where $B$ runs through the bounded subsets of $V$ and $U$ through the open subsets of $K.$  
\begin{lem} \label{lem:compact-open}
	Let $V$ be a classical Montel space over $K,$ let $V'_{b}$ (resp. $V'_{c}$) be its continuous dual $\Hom_{K}(V,K)$ equipped with the strong (resp.  compact-open) topology) then the continuous bijection 
	$$V'_b \to V'_c$$ is a homeomorphism.
\end{lem}
\begin{proof}
	In a Montel space every bounded set is compactoid. By \cite[Theorem 3.8.3]{PGS} over our locally compact $K$ every complete compactoid is compact. Now let $B$ be a bounded subset of $V$ and consider for $\varepsilon>0$ the open subset of $V'_b$ given by $\mathfrak{B}(B,\varepsilon):=\{f \in V' \mid \abs{f(B)}\leq \varepsilon\}.$ Then the closure of $B$ in $V$ is a complete compactoid by the Montel property and hence compact. We get $\mathfrak{B}(B,\varepsilon) = \mathfrak{B}(\overline{B},\varepsilon)$ by continuity and hence that $\mathfrak{B}(B,\varepsilon)$ is open in the compact-open topology. This concludes the proof, as $B$ was arbitrary.
\end{proof}
\begin{lem}
	We have 
	$$D^{la}_{cond}(G,K) = \underline{D^{la}(G,K)}$$ and $D^{la}_{cond}(G,K)(*)= D^{la}(G,K)$ as topological vector spaces. 
\end{lem}
\begin{proof}Write $V:=C^{an}(G,K) = \varinjlim V_n$ as an inductive limit of Banach spaces with compact transition maps. By Example \ref{ex:regular examples} and the proof of Lemma \ref{lem:regularLF} $V$ is lccg Hausdorff. 
	By (a $K$-linear version of) \cite[Proposition 4.2]{scholze2019condensed} we have $\underline{\Hom}_{\underline{K}}(\underline{V},\underline{K}) = \underline{\Hom_{K,cts}(V,K)}.$
	Where $V':=\Hom_{K,cts}(V,K)$ carries the compact-open topology, which we shall denote $V'_c.$ 
	We thus have  $D^{la}_{cond} = \underline{V'_c}.$
	It suffices to show that $V'_c$ is equal to $V'$ with the strong topology, which we shall denote by $V'_b.$ Because every compact set is bounded, it is easy to see, that the natural bijective map $V'_b \to V'_c$ from the strong dual to the compact-open dual is continuous.  By \cite[Theorem 11.3.5]{PGS} $V$ is a Montel space and we can conclude by Lemma \ref{lem:compact-open} that the strong topology and the compact-open topology agree in the present case. 
\end{proof}

For $W$ a $D^{la}_{cond}(G,K)$-module, we define $$H^i_{an}(G,W):= \underline{\operatorname{Ext}}^i_{D^{la}_{cond}(G,K)}(\underline{K},W)(*).$$
Note that if a $D^{la}_{cond}(G,K)$-module is of the form $\underline{W}$ for a complete locally convex $K$-vector space $W$ with a continuous $D^{la}(G,K)$-module structure this recovers the notion of analytic cohomology used in \cite{Kohlhaase}, namely $H^i_{an}(G,W):=\operatorname{Ext}^i_{D^{la}(G,K)}(K,W).$ Indeed, by \cite[Remark 6.4.]{Kohlhaase} the trivial module admits a projective resolution $P^\bullet$ with finitely generated projective terms and strict maps. Applying the functor $\underline{(-)}$ produces a projective resolution $\underline{P^\bullet}$ of $\underline{K}$ (cf. \cite[Proposition 2.2.1]{jacinto2023solidlocallyanalyticrepresentations}). 
\subsection{Classical and condensed analytic vectors}
\label{sec:proanalytic}
We define the classical (derived) analytic vectors of a Banach space $V$ with a continuous $G$-action as
$$V^{c-R^ila}:=H^i_{cts}(G,C^{an}(G,V)_{(1,3)}),$$
where $C^{an}(G,V) := C^{an}(G,K)\hat{\otimes}_K V.$
Note that a Fréchet space can always be written as a countable projective limit of Banach spaces with dense transition maps.  
For a Fréchet space $V = \varprojlim_m V_m,$ written as a countable projective limit of Banach spaces with dense transition maps we define the classical pro-analytic vectors as 
$V^{c-pa}:= \varprojlim_m V_m^{c-R^{0}an}.$ We extend these definitions to countable colimits of Banach spaces (resp. Fréchet spaces) by taking colimits, provided that each term in the colimit is $G$-stable. This is not automatic in general and will be assumed implicitly in this context. This agrees with the notion of Berger--Colmez and is independent of the choice of presentation. In the proofs of our main results we will avoid pro-analytic vectors by working exclusively with LB-spaces. Nonetheless we include the following paragraph on pro-analytic vectors for the convenience of the reader. 
We say a morphism of condensed sets has dense image, if the induced morphism of the underlying topological spaces has dense image.
For a condensed LF space $V$ we define analogously $V^{Rpa}:= \varinjlim_n\mathbf{R}\varprojlim_m V_{n,m}^{Rla}.$  Note that $V^{Rpa}$ depends on the presentation of $V$ in the form $V= \varinjlim_n \varprojlim_m V_{n,m}.$ We will implicitly fix a presentation of $V.$ We are not sure whether $V^{Rpa}$ only depends on $V.$ We will however prove that its analytic cohomology is equal to the condensed cohomology of $G$ with values in $V$ and hence is independent.  
\begin{thm}\label{thm:proancohomology} Let $V = \varinjlim_n \varprojlim_m V_{n,m}$ be a (condensed) LF-space written as a colimit of Fréchet spaces $V_n:= \varprojlim_{m}V_{n,m}$ written as a projective limit of Banach spaces with dense transition map equipped with compatible actions of a $p$-adic Lie group $G.$ 
	Then 
	$$\mathbf{R}\Gamma_{\text{cond}}(G,V)=\mathbf{R}\Gamma^{la}(G,V^{Rpa}) = \varinjlim_n\mathbf{R}\varprojlim_m \mathbf{R}\Gamma^{la}(G,V_{n,m}^{Rla}).$$
\end{thm}
\begin{proof}
	We have 
	\begin{align} 
		\mathbf{R}\Gamma^{la}(G,V^{Rpa}) &=  \mathbf{R}\Hom_{D^{la}_{\text{cond}}(G,K)}(K,\varinjlim_n\mathbf{R}\varprojlim_m  V_{n,m}^{Rla})\\ 
		&= \label{eq:compactness} \varinjlim_n \mathbf{R}\Hom_{D^{la}_{\text{cond}}(G,K)}(K,\mathbf{R}\varprojlim_m V_{n,m}^{Rla})\\
		&= \label{eq:commutativityRlim} \varinjlim_n\mathbf{R}\varprojlim_m \mathbf{R}\Hom_{D^{la}_{\text{cond}}(G,K)}(K,V_{n,m}^{Rla})\\
		&= \label{eq:JCapply} \varinjlim_n\mathbf{R}\varprojlim_m \mathbf{R}\Gamma_{\text{cond}}(G,V_{n,m})\\
		&= \label{eq:same} \mathbf{R}\Gamma_{\text{cond}}(G,V)
	\end{align}
	using compactness of $K$ as a $D^{la}_{\text{cond}}(G,K)$-module  in \eqref{eq:compactness} (cf.\ proof of \cite[Proposition 3.1.7]{jacinto2023solidlocallyanalyticrepresentations}), commuting $\mathbf{R}\Hom$ with $\mathbf{R} \varprojlim$ in \eqref{eq:commutativityRlim}, applying Theorem \ref{thm:analyticcondcomp} term wise in \eqref{eq:JCapply}. The last equation \eqref{eq:same} follows by the same reasoning as in \eqref{eq:compactness} and \eqref{eq:commutativityRlim}.
\end{proof}
Theorem \ref{thm:proancohomology} shows that $\varinjlim_n\mathbf{R}\varprojlim_m \mathbf{R}\Gamma^{la}(G,V_{n,m}^{Rla})$ is independent of the chosen presentation. Let us denote by $V^{R^ipa}$ the $i$-th cohomology of $V^{Rpa}.$

\begin{lem}
	\label{lem:mittaglefflerderivedanalytic}
	Let $V = \varinjlim_n V_n = \varinjlim_n \varprojlim_m V_{n,m}$ be as in Theorem \ref{thm:proancohomology}.
	\begin{enumerate}[(i)]
		\item If $R^1\varprojlim_{m}((V_{m,n})^{R^{i-1}la}) =0$ then $V_{n}^{R^ipa} = \varprojlim V_{m,n}^{R^ila}.$
		\item If the first point holds for every $n \gg 0$ then $V^{R^ipa} = \varinjlim_n\varprojlim V_{m,n}^{R^ila}.$
	\end{enumerate}
\end{lem}
\begin{proof}
	By analogue of \cite[\href{https://stacks.math.columbia.edu/tag/07KY}{Tag 07KY}]{stacks-project} for condensed abelian groups we have
	a family of short exact sequences 
	$$0 \to R^1\varprojlim_{m}((V_{m,n})^{R^{i-1}la}) \to (\mathbf{R}^i\varprojlim_m(V_{m,n})^{Rla}) \to \varprojlim_{m}((V_{m,n})^{R^ila}) \to 0.$$ 
	Using that $(V_{m,n})_m$ has dense transition maps one gets $\mathbf{R}\varprojlim_mV_{m,n} = \varprojlim_mV_{m,n}[0].$
	
	Hence {\it{(i)}} follows.  Point {\it{(ii)}} follows from {\it{(i)}} by taking colimits. 
\end{proof}
\begin{cor}\label{cor:proanalyticvectors}
	Let $V=\varinjlim_n \varprojlim_m V_{m,n}$ be as in Theorem \ref{thm:proancohomology}. Then $V^{R^0pa} = \varinjlim_n\varprojlim_{m}V_{m,n}^{R^0la}.$ 
\end{cor}
\begin{proof}
	We can apply Lemma \ref{lem:mittaglefflerderivedanalytic} to $i=0.$ Indeed,  $(V_{n,m})^{Rla}$ is concentrated in non-negative degrees for every $n$ which means that the $\varprojlim^1$-terms vanish for trivial reasons. 
\end{proof}
\begin{lem}\label{lem:regularanalyticvectors} 
	Condensed (pro-)analytic vectors and classical (pro-)analytic vectors are related as follows: 
	\begin{enumerate}[(i)]
		\item Let $V$ be a classical Banach space, then $\underline{V}^{R^ila}(*)=V^{c-R^ila}.$
		\item Let $V = \varprojlim V_m$ be a classical Fréchet space then 
		$$\underline{V}^{R^0pa}(*)=\varprojlim_m V_m^{c-R^0la} = V^{c-pa}.$$ 
		\item Let $V = \varinjlim_nV_n$ be a classical regular, complete LF-space  then  $$\underline{V}^{R^0pa}(*) = \varinjlim_n V_n^{c-pa} = V^{c-pa}.$$
	\end{enumerate}
\end{lem}
\begin{proof}
	The first point follows from Theorem \ref{thm:comparison} using {\it{(iv) and (v)}} of Lemma \ref{lem:tensorproduct}. For the second point apply Lemma \ref{lem:mittaglefflerderivedanalytic} in degree $i=0,$ where the $\lim^1$-terms are automatically $0.$
	For the third point by Lemma \ref{lem:regularLF} we have $\underline{V} = \varinjlim \underline{V_n}.$ Using this and that $\ZZ[*]$ is compact projective, reduces to the case of {\it{(ii)}}.  The equality $\varinjlim_n V_n^{c-pa} = V^{c-pa}$ holds by definition.
\end{proof}
This clarifies the relationship with pro-analytic vectors in the sense of \cite{bergercolmez2016theoriedesen}.
\subsection{Group cohomology}
The following lemma is a slight generalization of the first point of \cite[Lemma 3.3]{colmez2023arithmeticdualitypadicproetale}.
\begin{lem} \label{lem:settocondvanishing}
	Let $V$ be a classical $\QQ_p$-Banach space with a continuous action of a profinite group $G,$ then
	\begin{enumerate}[(i)]
		\item
		$$\mathbf{R\Gamma}_{\text{cond}}(G,\underline{V}) = \underline{\mathbf{R\Gamma}_{\text{cts}}(G,V)},$$ where $\mathbf{R\Gamma}_{\text{cts}}(G,V)$ denotes the standard continuous cochain complex whose terms are equipped with the compact open topology. 
		\item
		If furthermore the differentials of $\mathbf{R\Gamma}_{\text{cts}}(G,V)$ are strict (for example if $H^i_{cts}(G,V)$ is finite dimensional) then 
		$H^i_{\text{cond}}(G,V) = \underline{H^i_{\text{cts}}(G,V)},$ where the right-hand side is equipped with the subquotient topology of the compact open topology.
	\end{enumerate}
\end{lem}
\begin{proof}
	The left-hand side is represented by the complex $\underline{\Hom}(\ZZ[G^\bullet],\underline{V}),$ which evaluates at a profinite set $S$ to $$\Hom(\ZZ[S]\otimes\ZZ[G^\bullet],\underline{V}) \cong \Hom(\ZZ[S\times G^{\bullet}],\underline{V})$$ using that $\underline{V}$ is solid. 
	Further we have 
	$$\Hom(\ZZ_p[S\times G^{\bullet}],\underline{V}) \cong \underline{\operatorname{Map}_{cts}(G^\bullet\times S,V)}$$ with the right hand side equipped with the compact open topology. 
	It remains to see that the adjunction 
	$\operatorname{Map}_{cts}(G^\bullet\times S,V) \cong \operatorname{Map}_{cts}(S,\operatorname{Map}_{cts}(G^\bullet ,V))$ is an isomorphism for the respective compact open topologies.
	It is bijective because $G^n$ is locally compact Hausdorff, and it is continuous because $S$ is Hausdorff. Since the space of maps from a compact space into a Banach space equipped with the compact open topology is a Banach space we can conclude by the open mapping theorem that the continuous bijection is a homeomorphism. 
\end{proof}

\begin{lem}\label{lem:vanishingcond} Let $K/\QQ_p$ be a finite extension. The following hold:
	\begin{enumerate}
		\item $\mathbf{R}\Gamma_{\text{cond}}(G_K, \underline{\Bdr/\Bdr^+})\simeq0.$
		\item $\mathbf{R}\Gamma_{\text{cond}}(H_K, \underline{\fB_e})\simeq  \underline{H^0(H_K,\fB_e)}[0].$

	\end{enumerate}
	
\end{lem}
\begin{proof}

	By writing $\underline{\Bdr^+/\operatorname{Fil}^k\Bdr^+} = \varinjlim \underline{\Bdr^+/\operatorname{Fil}^k\Bdr^+(-i)},$ commuting with direct limits using compactness of the condensed set $\ZZ[G_K]$ it suffices to show $\mathbf{R}\Gamma_{\text{cond}}(G_K,\underline{\Bdr^+/\operatorname{Fil}^k\Bdr^+(-i)})\simeq0.$
	For the vanishing of the classical group cohomology is given by \ref{lem:classicalvanishing}. Since  $\mathbf{R}\Gamma_{cts}(G_K,\Bdr^+/\operatorname{Fil}^k\Bdr^+(-i))$ is an acyclic complex of Banach spaces the differentials are strict and we can apply the second part of Lemma \ref{lem:settocondvanishing} to conclude the vanishing of the condensed group cohomology. 
	The second case is treated analogously by first writing $\fB_e$ as a (strict) colimit of Banach spaces. By the classical vanishing results in Lemma	\ref{lem:vanishingHLadd} each complex is concentrated in degree $0$ and has strict differentials by a standard Open Mapping Theorem argument.
\end{proof}

\section{Computations of analytic vectors of $\fB_e$ and $\Bdr.$}
\subsection{Passing to $\QQ_p$-analytic vectors}
Let $E/\QQ_p$ be finite and
let $G_K \subseteq G_E$ be an open subgroup. 
Let $W_{\mathrm{dR}}^+$ be a finitely generated free $\Bdr^+$-module with a semi-linear action of $G_K.$
 Let $W_{\mathrm{dR}}:=W_{\mathrm{dR}}^+[1/t_{E}].$
We consider the objects $\mathcal{W}_{\mathrm{dR}}^{(+)}:=(W_{\mathrm{dR}}^{(+)})^{H}$ which we view as locally convex topological vector spaces with an action of the open subgroup $\Gamma_{K,E}$ of the $p$-adic Lie group $o_E^\times \cong \Gamma_{E,E} = G_E/H_{E,E}.$ Subsequently we will view $\Gamma:=\Gamma_{K,E}$ as an $[E:\QQ_p]$-dimensional $\QQ_p$-analytic group. 
For a complete topological vector space $\mathcal{W}$ with an action of $\Gamma,$ we will write $\mathcal{W}^{R^ila}:=\underline{\mathcal{W}}^{R^ila}(*)$ to keep notation light. 

\begin{prop} \label{prop:vanishingvectors} We have 

	$$(\mathcal{W}_{\mathrm{dR}}^+/\operatorname{Fil}^j)^{R^ila}=0$$ for every $j\geq 0$ and every $i \geq 1.$

\end{prop}
\begin{proof}Let $F= \CC_p^{H_{K,E}}.$
	By \cite[Remark 5.7]{porat2024locally}, the higher locally analytic vectors of any finite-dimensional semi-linear $F$-representation vanish. This implies the claim for $j=1$ because $\mathcal{W}_{\mathrm{dR}}^+/t_{L}\mathcal{W}_{\mathrm{dR}}^+$ is a finite dimensional $\CC_p^{H_{K,E}}= F$-representation. 
	For $j>1$ one concludes inductively using the short exact sequence 
\begin{equation}\label{eq:inductionderham}
	0 \to t_{E}^j\mathcal{W}_{\mathrm{dR}}^+/\mathcal{W}_{\mathrm{dR}}^+t^{j+1}_{E} (\chi_{LT}^{-j}) \to \mathcal{W}_{\mathrm{dR}}^+/\mathcal{W}_{\mathrm{dR}}^+t^{j+1}_{E} \to \mathcal{W}_{\mathrm{dR}}^+/t_{LE}^j\mathcal{W}_{\mathrm{dR}}^+ \to 0
\end{equation}
	The twist ensures that the sequence is $\Gamma$-equivariant and one can apply  \cite[Remark 5.7]{porat2024locally} to the left term to conclude the vanishing of the derived analytic vectors in the middle. 
 
\end{proof}

\begin{lem} \label{lem:BCvectors} Let $E=\QQ_p,$ $K/\QQ_p$ finite, and let $K_\infty:=K(\zeta_{p^{\infty}}).$ We have

		$$((\Bdr^+/\operatorname{Fil^i}\Bdr^+)^{H_K})^{R^0-la} =K_\infty \llbracket t_{\QQ_p} \rrbracket/(t_{\QQ_p}^i).$$

\end{lem}
\begin{proof}
	By Lemma \ref{lem:regularanalyticvectors} this boils down to computing classical analytic vectors. 
	See \cite[Proposition 2.6]{porat2022lubin} for a proof in the case $K=\QQ_p.$ The same proof works in the above case using \cite[Théorème 3.2]{bergercolmez2016theoriedesen} to start the inductive argument. The assumption that $K/\QQ_p$ is Galois is not required (as explained in \cite{porat2022lubin}).
\end{proof}

\begin{lem} 
	\label{lem:analyticvectorsofBe}Let $E=\QQ_p$ and let $K/\QQ_p$ be finite. Let $H = G_K\cap H_{\QQ_p}.$ Then
	the natural maps 
	$$\Q_pt^n_{\QQ_p} \to [((\fB_\mathrm{cris}^+)^{H})^{\varphi=p^n}]^{R^0la}$$
	and 
	$$\Q_p \to (\fB_e^{H})^{R^0la}$$
	are isomorphisms.
\end{lem} 
\begin{proof}
	See \cite[Proposition 4.8 and Remark 4.9]{porat2024locally}.
\end{proof}
\begin{rem}  For general $E$ the group $\Gamma_{K,E}$ is an $[E:\QQ_p]$-dimensional Lie group the derived analytic vectors with respect to $\Gamma_{K,E}$ could be non-zero beyond cohomological degree $1.$ Note however that Proposition \ref{prop:vanishingvectors} works in this generality. Furthermore, because the Lubin--Tate extension $E_\infty/E$ contains an unramified twist of the cyclotomic extension \cite[Theorem C]{porat2024locally} gives some control of derived analytic vectors in this context. This will be explored in future work. 
\end{rem}
\section{Main results.}
Let $K/\QQ_p$ be finite and let $G=G_K \subseteq G_{\QQ_p}$. Let $H = G_{\QQ_p}\cap G$ and $\Gamma:= G/H.$
Our strategy revolves around using the following Lemma. 
\begin{lem} \label{lem:allquasiisos}
	Let $V$ be a solid representation of $G$. Then
	\begin{align*}
	\mathbf{R}\Gamma_{\text{cond}}(G,V)&\simeq
	\mathbf{R}\Gamma_{\text{cond}}(\Gamma,\mathbf{R}\Gamma_{\text{cond}}(H,V))\\
	&\simeq \mathbf{R}\Gamma_{\text{cond}}(\Gamma,\mathbf{R}\Gamma(H,V)^{Rla}). 	\end{align*}
\end{lem}
\begin{proof}
	The first quasi-isomorphism is Hochschild--Serre, which for condensed groups follows formally by composition of derived functors using that the functor $\underline{\Hom}_{\mathbb{Z}[\underline{H}]}(\ZZ,-)$ preserves injectives by nature of being right adjoint to the exact forgetful functor from $\Z[\underline{\Gamma}]$-modules  $\Z[\underline{G}]$-modules, the second quasi-isomorphism is Theorem \ref{thm:analyticcondcomp}.
\end{proof}

Recall that the Lie-algebra of $\Gamma$ is one dimensional generated by an element $\nabla.$ Since the Lie algebra can be embedded into the distribution algebra, we get an action of $\nabla$ on any $D^{la}(\Gamma,\QQ_p)$-module. For $V$ be a solid vector space with derived locally analytic $\Gamma$-action 
we denote by $H^i_{\nabla}(V)$ the cohomology of the complex $V \xrightarrow{\nabla}V.$  We recall that $H^i_{cond}(\Gamma,V) = \underline{\operatorname{Ext}}^i_{D^{la}_{\text{cond}}(\Gamma,\QQ_p)}(\QQ_p,V) \cong H^0(\Gamma,H^i_\nabla(V))$  ( see  \cite[Theorem 1.7]{jacinto2022solidlocallyanalyticrepresentations} and \cite[Theorem 4.10]{Kohlhaase} for the corresponding classical statement).  Clearly, $H^i_{an}(\Gamma,-)$ vanishes for $i>1.$

Even for $E= \QQ_p,$ we have higher derived analytic vectors for the period ring $\fB_e= \Bcris^{\varphi=1}.$ 

\begin{thm} \label{thm:cohomologyofR1} Let $R:=\underline{H^0(H,\fB_e)}^{R^1-la}.$ Then $$H^0_{\text{cond}}(\Gamma,R) \cong \underline{\QQ_p}^{[K:\QQ_p]}$$ and $$H^1_{\text{cond}}(\Gamma,R)=0.$$
	In particular, $R$ is non-zero.
\end{thm}
\begin{proof}
	Consider the fundamental exact sequence 
	$$0 \to \QQ_p \to \fB_e \to \fB_{dR}^+/\fB_{dR} \to 0.$$ Since it is strictly exact we get an exact sequence of condensed sets. Applying $\mathbf{R}\Gamma_{\text{cond}}(G,-)$ and using Lemma \ref{lem:vanishingcond} we get 
	$\mathbf{R}\Gamma_{\text{cond}}(G,\underline{\QQ_p}) \simeq \mathbf{R}\Gamma_{\text{cond}}(G,\underline{\fB_e}) $
	The left hand side is, by Lemma \ref{lem:settocondvanishing} obtained from the continuous group cohomology of $G$ with values in $\QQ_p.$ We now rewrite the right hand side as 
	$\mathbf{R}\Gamma_{\text{cond}}(G,\underline{\fB_e}) \simeq \mathbf{R}\Gamma_{\text{la}}(\Gamma,\mathbf{R}\Gamma(H,\underline{\fB_e})^{Rla}).$ 
By truncating $\mathbf{R}\Gamma(H_K,\underline{\fB_e})^{Rla}$ we get a distinguished triangle 
	$$((\underline{\fB_e})^H)^{la} \to \mathbf{R}\Gamma(H,\underline{\fB_e})^{Rla} \to R[1] \to ((\underline{\fB_e})^H)^{la}[1].$$
	Using $\QQ_p \simeq ((\underline{\fB_e})^H)^{la}$ and taking $\Gamma$-cohomology we get a long exact sequence which splits into short exact sequences as follows:
	\[0 \to \underline{\QQ_p} \to \underline{\QQ_p} \to 0 \to 0\]
	\[0 \to \underline{\QQ_p} \to H^1_{\text{cond}}(\Gamma_K,\mathbf{R}\Gamma(H,\underline{\fB_e})^{Rla}) \to H^0_{\text{cond}}(\Gamma_K,R) \to 0\]
	\[ 0 \to H^2_{\text{cond}}(\Gamma,\mathbf{R}\Gamma(H,\underline{\fB_e})^{Rla}) \to H^1_{\text{cond}}(\Gamma,R) \to 0.\]
	To compute this splitting we used that $\QQ_p$ (viewed as a representation of $\Gamma$) is admissible, hence has no higher derived analytic vectors. The $\Gamma$-cohomology can then be computed from the $\nabla$-cohomology (with $\nabla \colon \QQ_p \to \QQ_p$ being the zero map). Now $\mathbf{R}\Gamma_{\text{cond}}(G,\underline{\QQ_p}) \simeq \mathbf{R}\Gamma_{\text{cond}}(\mathbf{R}\Gamma(H,\underline{\fB_e})^{Rla})$ by Lemma \ref{lem:allquasiisos} and 
	 using Lemma \ref{lem:settocondvanishing} applied to  we conclude 
	$$0 = \underline{H^2_{cts}(G,\QQ_p)} = H^2_{\text{cond}}(\Gamma,\mathbf{R}\Gamma(H,\underline{\fB_e})^{Rla}) =H^1_{\text{cond}}(\Gamma,R) $$ and 
	$H^0_{\text{cond}}(\Gamma,R)$ is the quotient of $$H^1_{\text{cond}}(\Gamma,\mathbf{R}\Gamma(H,\underline{\fB_e})^{Rla}) = \underline{H^1_{cts}(G,\QQ_p)} \cong \underline{\QQ_p^{[K:\QQ_p]+1}}$$ by a subspace isomorphic to $\underline{\QQ_p}$ hence the claim.
\end{proof}
 We now revisit the analytic vectors of $((\Bcris^+)^{\varphi=p^n})^{H}$ which were initially considered in \cite{bergercolmez2016theoriedesen}.

\begin{thm}\label{thm:BC16} 
	Let $n\geq 1$ and let $U_n:=(\Bcris^+)^{\varphi=p^n}.$  Then there exists a short exact sequence 
	$$0 \to {K(\zeta_{p^{\infty}})}\llbracket t_{\QQ_p}\rrbracket/(t_{\QQ_p})^n \xrightarrow{\delta^{la}} H^1_{cts}(H,\QQ_p(n))^{c-R^0la}\to H^0(H,U_n)^{c-R^1la} \to 0.$$ 
	The map $\delta^{la}$ is equal to the restriction of boundary map $$\delta \colon H^0(H,\Bdr^+/t_{\QQ_p}^n\Bdr^+) \to H^1_{cts}(H,\QQ_p(n))$$ (obtained by applying $H$-invariants to the fundamental exact sequence) to the subspace of analytic vectors. 
\end{thm}
\begin{proof}
	We consider the short exact sequence
	$$0 \to \QQ_p(n) \to U_n \to \Bdr/t^n\Bdr \to 0$$ taking $H_K$-cohomology we get a short exact sequence 
	$$0 \to \QQ_p(n) \to H^0(H,U_n) \to H^0(H,\Bdr/t^n_{\QQ_p}\Bdr) \to H^1(H,\QQ_p(n)) \to 0$$ which is an exact sequence of Banach spaces and by the Open Mapping Theorem even strict. We get an exact sequence of condensed sets
		$$0 \to \underline{\QQ_p(n)} \to \underline{H^0(H,U_n)} \to \underline{H^0(H,\Bdr/t^n_{\QQ_p}\Bdr)} \to \underline{H^1(H,\QQ_p(n))} \to 0.$$
	The $\Gamma_K$-representations $\QQ_p(n)$ and $H^1(H,\QQ_p(n))$ are admissible in the sense that their duals are finitely generated over the completed group ring $\ZZ_p\llbracket \Gamma_K\rrbracket[1/p]$, as a consequence of \cite[Remark 4.47 and Proposition 4.48]{jacinto2022solidlocallyanalyticrepresentations} their condensed sets have no higher derived locally analytic vectors. It is easy to see that
		$R:= \underline{H^0(H,U_n)}^{R^1-la} \cong (\underline{H^0(H,U_n)/\QQ_p})^{R^1-la}$ using that we have $(\underline{H^0(H,U_n)/\QQ_p})^{R^0-la}=0$ and $\underline{H^1(H,\QQ_p(n))}^{R^1-la}=0$ by admissibility, we get a long exact sequence
	$$0 \to \underline{H^0(H,\Bdr/t^n_{\QQ_p}\Bdr^+)}^{R^0-la} \to \underline{H^1(H,\QQ_p(n))}^{R^1-la} \to R \to X \to 0,$$ where $X:=\underline{H^0(H_,\Bdr/t^n_{\QQ_p}\Bdr^+)}^{R^1-la}.$
	The underlying set $X(*)$ is trivial by Proposition \ref{prop:vanishingvectors}. Hence evaluation at $*$ gives the desired result using the explicit description in Lemma \ref{lem:BCvectors}. The statement about the map $\delta^{la}$ is clear by construction.
	
\end{proof}
As a corollary we obtain an explicit description of the underlying set of $\underline{H^0(H_K,\fB_e)}^{R^1la}.$
\begin{cor} \label{cor:colimBe}
	By writing $\fB_e = \varinjlim_n t^{-n}U_n$ we obtain 
	$$H^0(H,\fB_e)^{c-R^1la} = \varinjlim_{n} H^0(H,U_n)^{c-R^1la}(-n).$$
\end{cor}
\begin{proof}
	We have $H^0(H,U_n)(-n) = H^0(H,t^{-n}U_n)$ as representations. Hence it suffices to show, that $W(n)^{Rla}\cong W^{Rla}(n)$ for any $W\in \operatorname{Rep}_{\QQ_{p,\blacksquare}}(\Gamma).$ This could be deduced from a general projection formula but in this case we derive it by a simple calculation. 
	To this end consider the map 
	$\alpha \colon \underline{C^{an}(\Gamma_K,\QQ_p)} \to \underline{C^{an}(\Gamma_K,\QQ_p) (-n)}$ induced by the map 
	$f \mapsto [x \mapsto \chi_{cyc}^{-n}(x)f(x)],$ which is a homeomorphism with continuous inverse $g\mapsto [x \mapsto \chi_{cyc}^{n}(x)g(x)].$   
	A small calculation shows that $\alpha$ is equivariant for the $(1,3)$-action and that the resulting isomorphism
	$$\mathbf{R}\Gamma_{\operatorname{cond}}(\Gamma,C^{la}(\Gamma,W)_{*_{1,3}})\otimes^{\mathbb{L}}_{\QQ_{p,\blacksquare},2}\underline{\QQ_p(-n)}\to \mathbf{R}\Gamma_{\operatorname{cond}}(\Gamma,C^{la}(\Gamma,W(-n))_{*_{1,3}}),$$ induced by the map 
	$\alpha \otimes \id_W,$ where the tensor product on the left-hand side is equipped with the diagonal action of the $2$-action, is equivariant with respect to the $2$-action on the right hand side. 
	Furthermore, we can omit the derived tensor product because the terms of the complex are solid and $\QQ_p(-n)$ is a Fréchet space using Lemma \ref{lem:tensorproduct}. 
\end{proof}
\begin{rem}~

	It would be interesting to investigate whether the condensed set $X$ in the proof of Theorem \ref{thm:BC16} is trivial. We can not directly apply Lemma \ref{lem:settocondvanishing}, as, e.g., $\widehat{K_\infty}^{Rla}(*)$ is concentrated in degree zero but not (in an obvious way,i.e., when written as a colimit of its $\Gamma_n$-analytic vectors for a fundamental system of open zero-neighbourhoods $(\Gamma_n)_n$ in $\Gamma_K$) a colimit of complexes of Banach spaces, which themselves are concentrated in degree zero but rather (cf. \cite[Proof of Theorem 5.1]{porat2024locally}) a given class in $H^1$ gets killed in a higher level of the colimit, but we do not know whether the transition maps are strict.

\end{rem}

\section*{Data availability statement}
There is no associated data.
\section*{Conflict of interest statement}
The author declares no competing interests.
\let\stdthebibliography\thebibliography
\let\stdendthebibliography\endthebibliography
\renewenvironment*{thebibliography}[1]{%
	\stdthebibliography{RJRC23}}
{\stdendthebibliography}

\bibliographystyle{amsalpha}

\bibliography{Literatur}
\end{document}